\documentclass{amsart}
\usepackage{fullpage}

\usepackage{amsfonts,amssymb,latexsym,amsthm,mathtools,amscd,gastex,yhmath}
\usepackage{verbatim}

\numberwithin{equation}{section}

\newcommand\dessins[3]{}

\newcommand\invis[1]{}
\newcommand\footinvis[1]{}
\newcommand\nouvellepage{}

\def\ul{\underline}
\def\bs{\setminus}

\def\Ra{\Rightarrow}
 
\def\Lgrh{\Longrightarrow} 
\def\Lra{\Leftrightarrow}

\def\bfp{{\mathbf p}}
\def\bfq{{\mathbf q}}
\def\bfa{{\mathbf a}}
\def\bfb{{\mathbf b}}
\def\bfr{{\mathbf r}}
\def\bfk{{\mathbf k}}

\def\iff{{\em if and only if}}
\def\wrt{{\em w.r.t.}}
\def\ie{{\em i.e.}} 
\def\eg{{\em e.g.}} 
\def\etc{{\em etc.}}
\def\cf{{\em cf.}} 

\def\N{{\mathbb N}}

\DeclareMathOperator{\id}{id}

\def\multordby{\cdot}
\def\multsetby{\times}

\def\uloplus{{\;\ul\oplus\;}}

\theoremstyle{plain}

\newtheorem{lem}{Lemma}[section]
\newtheorem{prop}{Proposition}[section]
\newtheorem{cor}{Corollary}[section]
\newtheorem{theo}{Theorem}[section]
\theoremstyle{definition}

\newtheorem{notat}{Notation}[section]

\title[Length of an intersection]{Length of an intersection}

\author[C. Delhomm\'e]{Christian Delhomm\'e}
\address{LIM-ERMIT, Universit\'e de la R\'eunion, Facult\'e des Sciences et Technologies - PTU, 2, rue Joseph Wetzel, 97490 Sainte-Clotilde, France}  
\email{delhomme@univ-reunion.fr}

\author [M. Pouzet]{Maurice Pouzet}
\address{ICJ, Math\'ematiques, Universit\'e Claude-Bernard Lyon1, 43 bd. 11 Novembre 1918, 69622 Villeurbanne Cedex, France and Mathematics \& Statistics Department, University of Calgary, Calgary, Alberta, Canada T2N 1N4} 
\email{pouzet@univ-lyon1.fr, mpouzet@ucalgary.ca }

\def\acknolDIMACOS%
{
{\em Acknowledgement~:}
Results of this paper have been presented at the International Conference on Discrete Mathematics and Computer Science (DIMACOS'11) organized by A.~Boussa\"{\i}ri, M.~Kabil, and A.~Taik  in Mohammedia (Morocco) May, 5-8, 2011. We are pleased to thank   the organizers for their warmful invitation.
}

\date{\today} 

\begin{document}
\maketitle


\begin{abstract}
A poset $\bfp$ is well-partially ordered (WPO) if all its linear extensions are well orders~;
the supremum of ordered types of these linear extensions is the {\em length}, $\ell(\bfp)$ of $\bfp$.
We prove that if the vertex set $X$ of $\bfp$ is infinite, of cardinality $\kappa$, 
and the ordering $\leq$ is the intersection of finitely many partial orderings $\leq_i$ on $X$,
$1\leq i\leq n$, 
then, letting $\ell(X,\leq_i)=\kappa\multordby q_i+r_i$, with $r_i<\kappa$, denote the euclidian division by $\kappa$ (seen as an initial ordinal) of the length of the corresponding poset~:%
%
%
\[
\ell(\bfp)<
\kappa\multordby\bigotimes_{1\leq i\leq n}q_i+
\Big|\sum_{1\leq i\leq n} r_i\Big|^+
\]
where $|\sum r_i|^+$ denotes the least initial ordinal greater than the ordinal $\sum r_i$.
This inequality is optimal (for $n\geq 2$).
\end{abstract}


\section{Introduction}

\subsection{Presentation of the result}

Let $\mathbf p:= (X, \leq)$ be an ordered set(poset). 
This poset is \emph{well founded} if every non-empty subset of the vertex set $X$ contains some minimal element. 
It is \emph{well partially ordered} (WPO for short) if,  in addition,  it has no  infinite antichain
\ie\ if every infinite set of vertices has comparable elements. 
Since its introduction by Erd\"os and Rado and by Higman~\cite{higman52} the notion of WPO has attracted considerable interest in various areas of mathematics and computer science
(\eg~\cite{robertson1}, \cite {lalement}).  
It was  observed by Wolk \cite{wolk} that a poset  $\mathbf p$ is WPO if and only \iff\ its linear extensions are all well orders. 
It has been proved by
de Jongh and Parikh~\cite{dejongh-parikh}
that  there is a largest ordinal type of these linear extensions, 
that we denote $\ell(\mathbf p)$ and call the \emph{length of $\mathbf p$} (or of its ordering $\leq$). 
The length behaves nicely \wrt\ to some poset and ordinal operations. 
For an example, in~\cite{dejongh-parikh},  de Jongh and Parikh extended  Carruth formulas~\cite{carruth} for direct sum and cartesian product  of well ordered chains, showing that if $\mathbf p$ and $\mathbf q$ are two WPO's then the lengths of their direct sum $\mathbf p\uplus \mathbf q$ and of their cartesian product $\mathbf p\times \mathbf q$ (which are WPO) satisfy~:
\begin{equation}\label{eq:dejongh}
\ell(\mathbf p \uplus \mathbf q)=\ell(\mathbf p)\oplus \ell(\mathbf q)
\text{~and }
\ell(\mathbf p\times \mathbf q)= \ell(\mathbf p)\otimes\ell(\mathbf q)
\end{equation}
where 
$\oplus$ and $\otimes$ denote the Hessenberg addition and multiplication, also called natural operations.  
Since then, the ordinal length of various WPO has be computed (\cite{schmidt,kriz,pouzet,rathjen}). 

In this note we consider posets $\mathbf p$ of which the ordering $\leq $ is the intersection of finitely many orderings $\leq_i$, $1\leq i\leq n$,  such that the  corresponding posets $\mathbf p_i$ are WPO. Since such a poset can be embedded into the direct product $\prod_{1\leq i\leq n} \mathbf p_i$,  it is WPO and it follows from de Jongh-Parikh formula~\eqref{eq:dejongh} above that $\ell(\mathbf p)\leq \bigotimes_{1\leq i\leq n}\ell(\mathbf p_{i})$. This upper bound is crude. 
The best upper bound is given in the next theorem: 

\begin{theo}\label {theo:main}
Given finitely many equipotent ordinals $\alpha_i$, $1\leq i\leq n$,  of cardinality $\kappa$,
consider their euclidian division by $\kappa$~:
$\alpha_i=\kappa\multordby q_i+ r_i$ with $r_i<\kappa$,
and the least initial ordinal $|\sum_{1\leq i\leq n} r_i|^+$ greater than $\sum_{1\leq i\leq n}r_i$. Then~:
\[
\kappa\multordby\bigotimes_{1\leq i\leq n}q_i+ \big|\sum_{1\leq i\leq n} r_i\big|^+
\] 
is the least strict upper bound of the lengths of the posets $\mathbf p=(X,\leq)$ of which the ordering $\leq$ is the  intersection of $n$ orderings $\leq_{i}$, $1\leq i\leq n$, such that each $(X,\leq_i)$ is a WPO of length $\alpha_i$.  
This strict least upper bound is unchanged if each ordering $\leq_i$ is required to be linear. 
\end{theo}


In particular~:

\begin{prop} If an ordering $\leq$ on a  set $X$ of cardinality $\kappa$ is the  intersection of $n$ orderings $\leq_i$, $1\leq i\leq n$, then the length of $\mathbf p:= (X, \leq)$ is at most  
$\kappa \multordby(\otimes_{1\leq i\leq n} q_i)$ provided that each  corresponding poset  $\mathbf p_i$ have length at most $\kappa\multordby q_i$.  This bound is attained for a family of $n$ linear orderings of types $\kappa\multordby q_i$. 
\end{prop} 

As a consequence~:

\begin{cor}\label{cor:length-sierp}
The length of a WPO $\mathbf p$ is an initial ordinal  $\kappa$ whenever the ordering  of $\mathbf p$ is the intersection of finitely many orderings  of length $\kappa$. 
\end{cor}

Indeed, the length of a poset $\mathbf p$ is trivially at least its cardinality (viewed as an initial ordinal), hence $\ell (\mathbf p)\geq\kappa$. Since in  the theorem above, $q_i= 1$ and $r_i=0$, $\kappa+1$ is a strict upper bound of the length of $\mathbf p$. Hence the length of $\mathbf p$ is $\kappa$. 

As another  consequence~:
 
\begin{cor} \label{cor:beta-omega}
Let $\mathbf p$ be  a poset and $\alpha$ be a countably infinite ordinal. Then $\ell (\mathbf p)=\alpha$ provided that the ordering is the intersection of an ordering of length $\omega$ and an ordering of length $\alpha$.
\end{cor}
 
Note that, according to Theorem~\ref{theo:main}, this result does not hold with $\omega_1$ in place of $\omega$ and $\alpha= \omega_1+\delta$ with $\omega\leq\delta<\omega_1$. 
Indeed, for each $\beta$ satisfying $\alpha\leq\beta\leq\omega_1$,
there a poset of length $\beta$ 
which is the  intersection of a linear order of type $\omega_1$ and a linear order of type $\alpha$. 

Theorem~\ref{theo:main} answers a question raised by Forster in~\cite{forster} in terms of the height of the tree of bad sequences of a poset $\mathbf p$ (see 
page 46). 
We recall that a sequence $x_0$, $x_1$, $\dots $ (finite or not ) of vertices of a poset $\mathbf p$ is called \emph{good}
if there are indices $i,j$ such that $i<j$ and $x_i\leq x_j$ and is called \emph{bad} otherwise. The set of all bad sequences of  $\mathbf p$ is denoted by $Bad (\mathbf p)$. If we compare sequences by extension, this set becomes 
 a tree with root $r$ the empty sequence. As it is easy to see, this tree has no infinite chain if and only if $\mathbf p$ is WPO, hence in this case $r$ has an height in $Bad(\mathbf p)$. It turns out that this height is equal to $\ell(\mathbf p)$ (\cite{kriz}). 
 
Instead of the length of orderings which are intersections of orderings, the height of orderings which are unions of non  necessarily transitive relations has been computed and an analogous formula obtained, see~\cite{blass-gurevich,delhomme,delhomme-treeaut}.

Corollary~\ref{cor:beta-omega} was conjectured by Forster~\cite{forster}~page 47. 
 
\subsection{Organisation of the paper}\label{section:organisation}
In the preliminary Section~\ref{section:prelim}, 
we fix our notation on ordinals and WPO and we recall basic facts on those.
The proof of Theorem~\ref{theo:main} is handled 
in Section~\ref{section:proof}.
We first introduce specific notation regarding our particular purpose,
and we restate our main result in this framework. 
We reduce this proof to the case $n=2$ (Section~\ref{section:reduction}),
and then the proof of this case $n=2$ is split into a minoration in Section~\ref{section:minoration},
and a majoration in Section~\ref{section:majoration}.
\invis{
See the end of Section~\ref{section:strategy} for a precise formulation.
}


Let us close this introductory section with some consequences of Corollary~\ref{cor:length-sierp} above, that were our initial motivation for the present work.
 
\subsection{Sierpinskisation}\label{section:sierp}

Let $\alpha$ be a countable order type and $\omega$ be the order type of non-negative integers.  A \emph{sierpinskisation} of $\alpha$ and $\omega$, or simply of $\alpha$,  is any poset $\mathbf s:=(X, \leq)$ such that the ordering on $X$ is the intersection of two linear orderings on $X$, one of type $\alpha$,  the other of type $\omega$. Such a  sierpinskisation can   be obtained from a bijective map $\varphi:\omega \rightarrow \alpha$, setting $X:=\N$ and $x\leq y$ if $x\leq y$ w.r.t. the natural ordering on  $\N$ and  $\varphi(x)\leq \varphi(y)$ w.r.t. the ordering of type  $\alpha$.

A consequence of  Corollary \ref{cor:beta-omega} is this:  

\begin{lem}\label{lem:sierp1} If $\alpha$ and $\alpha'$ are two countable ordinals with $\alpha<\alpha'$ then  no  sierpinskization of $\alpha'$ can be embedded in any sierpinskization of $\alpha$. 
\end{lem}
\begin{proof}If  a sierpinskization $\mathbf s'$ of $\alpha'$ is embeddable into a sierpinskization $\mathbf s$ of $\alpha$ then   by Corollary~\ref{cor:beta-omega} and (\ref{eq:restriction}) of Section~%
\ref{section:WPO}, 
$\alpha'=\ell(\mathbf s') \leq \ell (\mathbf s)=\alpha$. 
\end{proof}
 
 In contrast, we recall the following result (Lemma 3.4.1 
of \cite{pouz-zagu84}).
\begin{lem}\label{lem:sierp2}
For every countable order type $\alpha$, if $\mathbf s$ and $\mathbf s'$ are  a sierpinskizations of $\alpha$ and $\omega.\alpha$ respectively, then $\mathbf s$ is embeddable into $\mathbf s'$. \end{lem}

As  a special case,  if $\alpha$ and $\omega.\alpha$ are equimorphic then  any two sierpinskizations of $\alpha$ are equimorphic. If $\alpha$ is an ordinal, the equimorphy of $\alpha$ and $\omega.\alpha$ amounts to  $\omega.ind(\alpha)=\alpha$ (where $ind(\gamma):= 0$ if  $\gamma = 0$ and otherwise $ind(\gamma ) := \delta$, 
where $\delta$ is the least non-zero ordinal such that $\gamma=\gamma' + \delta$ for some $\gamma'$). If $\beta\geq \omega$ then $\alpha:= \omega^\beta$ satisfies this condition. Hence:

\begin{cor} \label{cor:omega1}The $\omega_1$-sequence $(\mathbf s_{\omega^\beta})_{\omega\leq \beta<\omega_1}$ of sierpinskizations of $\omega^\beta$ is strictly increasing w.r.t. embeddability. 
\end{cor}

To a poset $\mathbf p:= (X, \leq)$, there are two natural posets associated with $\mathbf p$, namely, the poset $I(\mathbf p)$ of  initial segments  of $\mathbf p$ and the poset $I_{<\omega}(\mathbf p)$ of finitely generated initial segments of $\mathbf p$, that is the finite unions of  principal initial segment of $\mathbf p$, sets of the form $\downarrow x:= \{y\in X: y\leq x\}$ for $x\in X$.   If $\mathbf {s}:= (X, \leq)$ is  a sierpinskization of  $\alpha$ then each principal initial segment of $\mathbf s$ is finite, hence $I_{<\omega}(\mathbf s)$  is a distributive lattice which is embeddable as a sublattice into $[\omega ]^{<\omega}$, the lattice of finite subsets of $\omega$, ordered by inclusion. If $\alpha$ is not an ordinal then $\mathbf s$ contains an infinite antichain. This implies that   $[\omega]^{<\omega}$ is also embeddable in $I_{<\omega}(\mathbf s)$ as a join-semilattice. For ordinals the situation is different: since $\mathbf s$ is WPO, $I_{<\omega}(\mathbf s)$ is WPO via  Higman's result \cite{higman52} on finite sequences, hence   $[\omega]^{<\omega}$ is not embeddable in $I_{<\omega}(\mathbf s)$.  
 
Similarly to Corollary  \ref{cor:omega1},  the  lattices $I_{<\omega}(\mathbf s_{\omega^\beta})$ form a $\omega_1$-chain w.r.t. join-semilattices embeddings and no member of this chain is embeddable in a previous member w.r.t. posets  embeddings.  Indeed:
 
 \begin{lem}
 Let $\alpha$, $\alpha'$ be two countably infinite  ordinals and $\mathbf s$,  $\mathbf s'$ be two sierpinskizations of $\alpha$ and $\alpha'$. 
If $\mathbf s$ is order-embeddable into $\mathbf s'$ then $I_{<\omega}(\mathbf s)$ is embeddable in  $I_{<\omega}(\mathbf s')$ as a join-semilattice, and this in turn implies that  $I(\mathbf {s})$ is embeddable in  $I (\mathbf s')$ as a join-semilattice.  If $I(\mathbf {s})$ is embeddable in  $I (\mathbf {s'})$ as a join-semilattice  then  $\alpha \leq\alpha'$. The converse of these implications holds if $\omega\cdot \alpha\leq \alpha'$. 
\end{lem}


\begin{proof} An order-embedding of $\mathbf s$ into $\mathbf s'$ yields a join-semilattice embedding of $I_{<\omega}(\mathbf s)$ into  $I_{<\omega}(\mathbf s')$. A join-semilattice embedding of $I_{<\omega}(\mathbf s)$ into  $I_{<\omega}(\mathbf s')$ yields too a  join-semilattice embedding of $I(\mathbf {s})$ into  $I (\mathbf s')$. According to Corollary \ref{cor:beta-omega},  $\ell(\mathbf s)=\alpha$, hence the maximum length of chains in $I(\mathbf s)$ is $\ell (\mathbf s)+1$, thus if  $I(\mathbf {s})$ is embeddable into   $I (\mathbf s')$ as a join-semilattice, it is embeddable as a poset and in this case $\ell (\mathbf s)+1\leq \ell (\mathbf s')+1$, thus $\alpha\leq \alpha'$. If  $\omega\cdot\alpha\leq \alpha'$,  then since $\mathbf s'$  induces a sierpinskization of $\omega\cdot\alpha$ and $\mathbf s$ embeds into this sierpinskization by Lemma  \ref{lem:sierp2},  it embeds into $\mathbf s'$. 
 \end{proof}

 Let  $\beta$ be an order type and  $\alpha:= \omega.\beta$. A  sierpinskization of $\alpha$ is called \emph{monotonic} if for every $\gamma\in \alpha$, the map $\varphi_{-1}$, once restricted to $\omega\times \{\gamma\}$ is monotonic. According to Lemma 3.4.3.  of   \cite{pouz-zagu84} two monotonic sierpinskization of $\alpha$ are equimorphic. Hence, we extend the conclusion of Corollary \ref{cor:omega1} to monotonic sierpinskizations of $\omega.\beta$ for every $\beta<\omega$ as well as to the distributive lattices they generate.
 
 For a detailed study of possible lengths of chains in algebraic lattices and sierpinskizations, see \cite{chakir} and \cite{chakir-pouzet1, chakir-pouzet2}. 

\acknolDIMACOS


\section{Preliminaries}\label{section:prelim}

\subsection{Ordinals and WPO. Notation and basic properties}\label{section:ordin-WPO}
See~\cite{fraissetr}.

\subsubsection{Ordinals}\label{section:ordin}

%
Given ordinals $\alpha$ and $\beta$~:
\begin{itemize}
  \item $|\alpha|$ denotes the cardinality of $\alpha$, considered as the least ordinal that is equipotent with $\alpha$.
  If $|\alpha|=\alpha$, then the ordinal $\alpha$ is {\em initial}.
  \item $\alpha^+$ denotes the least initial ordinal greater than $\alpha$ 
  (its so called Hartog).%
\footinvis{
The least ordinal greater than $\alpha$ failing to be equipotent to it.
}
\item If $\alpha\leq\beta$, then $(-\alpha)+\beta$ denotes the only ordinal $\gamma$ such that $\beta=\alpha+\gamma$.
Thus for a non zero ordinal $\alpha$, $(-1)+\alpha$ is equal to $n-1$ if $\alpha$ is an integer $n$, and it is equal to $\alpha$ itself if $\alpha$ is infinite.
  \item For each non-zero ordinal $\gamma$,
  $\alpha=\delta\multordby q_\delta(\alpha)+r_\delta(\alpha)$ denote the euclidian division of $\alpha$ by $\delta$,
  that is characterized by the remainder being less than $\delta$~: $r_\delta(\alpha)<\delta$.
\end{itemize}
%

$\oplus$ and $\otimes$ denote the natural addition and multiplication on ordinals (also called Hessenberg operations).
\invis{
Il conviendrait de pr\'eciser les propri\'et\'es qu'on utilisera, et \'eventuellement de faire une sous-sous-section sur ces op\'erations naturelles.
}

\medskip

For every set $A$ of ordinals, let $\sup A$ denote its {\em supremum},
\ie\  the least ordinal greater than or equal to every element of $A$,
and let $\sup^+A$ denote its least strict upper-bound,
\ie\ least ordinal greater than every element of $A$.
In particular, $\sup^+\varnothing=\sup\varnothing=0$ and $\sup^+A=\sup\{\alpha+1:\alpha\in A\}$.

\subsubsection{Posets}\label{section:posets}

We view a {\em poset} (or {\em ordered set}) $\bfp$ 
as a pair $(X,\leq)$,
where $\leq$ is the {\em ordering} of the poset and $X$ the {\em vertex set}. 
Thus an ordering of $X$ is a set of ordered pairs of elements of $X$.
If needed, we may denote by $\leq_\bfp$ the ordering of a poset $\bfp$.

Given two posets $\bfp:=(X,\leq_\bfp)$ and $\bfq:=(X,\leq_\bfq)$ with the same vertex set $X$~:
\begin{itemize}
\item
$\bfp\sqsubseteq\bfq$ means that $\bfq$ is an (edge)-{\em extension} of $\bfp$, 
\ie\ that $x\leq_\bfp y\Ra x\leq_\bfq y$.
\item
let $\bfp\sqcap\bfq:=(X,\leq_\bfp\cap\leq_\bfq)$ denote the intersection poset on $X$.
Likewise is defined the intersection of any number of posets on a same set.
\end{itemize}

An {\em initial segment} $Y$ of a poset $\bfp:=(X,\leq)$ is any set $Y\subseteq X$ of vertices such that
$x\leq x'\in X'\Ra x\in X'$.
Final segments are defined likewise.  

\medskip

An {\em order type} is an isomorphy type of posets.
The {\em order type} of a well order $\bfp$, is identified with the unique ordinal it is isomorphic to, and
will be denoted $\tau(\bfp)$.

\medskip

When we write that an application $f$ between ordered sets is {$\leq$-increasing},
we mean that $x\leq y\Ra f(x)\leq f(y)$.

\begin{notat}
If $\prec$ is a binary relation on a set $X$, \eg\ an ordering $\leq$, or the corresponding strict ordering $<$, or $\not\geq$, \etc, then for each $x\in X$ and $Y\subseteq X$, we let~: 
\[
\{Y\prec x\}:=\{y\in Y:y\prec x\}.
\]
If in addition $\bfr$ is a relational structure with vertex set $X$, 
\eg\ of the form $(X,\leq)$, or $(X,>)$, or $(X,\not\geq)$, \etc,
then we let~: 
\[
\{\bfr\prec x\}:=\bfr\restriction\{y\in X:y\prec x\}
\]
denote the corresponding induced substructure.
\end{notat}

\subsubsection{WPO}\label{section:WPO}
Basics on WPO can be found in~\cite{milnerwqobqo}.
Recall that, if a poset $\bfp$ is a WPO, then $\ell(\bfp)$ denotes its length.
Notice that if $\bfp$ is a well order then $\ell(\bfp)=\tau(\bfp)$.

WPO also admit the following characterization~:
a poset $\bfp$ is a WPO \iff\ the collection of its initial segments is well founded under inclusion,
and in this case the length of $\bfp$ is equal to the height of its vertex set.

We shall use the following observations.
Given a WPO $\bfp:=(X,\leq)$~:

\begin{itemize}
\item
The length may be inductively computed~\cite{dejongh-parikh}~:
\begin{equation}\label{eq:induction}
\ell(\bfp)=\sup^+_{x\in X}\ell(\{\bfp\not\geq  x\}).
\end{equation}
\item  If $X'\subseteq X$ then~: 
\begin{equation}\label{eq:restriction} 
\ell(\bfp\restriction X')\leq\ell(\bfp) 
\leq\ell((\bfp\restriction X')\uplus(\bfp\restriction X\bs X'))
=\ell(\bfp\restriction X')\oplus\ell(\bfp\restriction X\bs X').
\end{equation}
The middle inequality follows from $\bfp$ being an edge-extension of $(\bfp\restriction X')\uplus(\bfp\restriction X\bs X')$, while the right-hand one is~\eqref{eq:dejongh}.
Incidentally, notice that if, in addition, $X'$ is an initial segment, then 
$\ell(\bfp\restriction X')+\ell(\bfp\restriction X\bs X')\leq\ell(\bfp)$.
\end{itemize}

We shall also need the following lemmas~:

\begin{lem}\label{lem:longcut} 
Given a WPO $\bfp$ of length $\alpha$, 
consider a decomposition $\alpha=\alpha'+\alpha''$ of this length.
There is a partition of the domain $X$ of $\bfp$ into an initial segment $X'$ and a final segment $X''$ of $\bfp$ such that $\ell(\bfa\restriction X')=\alpha'$ and $\ell(\bfa\restriction X'')=\alpha''$.
\end{lem}

\begin{proof}
Given a linear extension $\bfa:=(X,\preccurlyeq)$ of type $\alpha$ of $\bfp$,
consider the initial segment $Y'$ of $\bfa$ of type $\alpha'$, 
and let $Y''$ denote the complementary final segment, of type $\alpha''$. 
Thus $\ell(\bfp\restriction Y')\geq\alpha'$ and $\ell(\bfp\restriction Y'')\geq\alpha''$. 
So consider an initial segment $X'$ of $\bfp\restriction Y'$ such that 
$\ell(\bfp\restriction X')=\alpha'$ and let $X'':=X\bs X'\supseteq Y''$.
In particular $\ell(\bfp\restriction X'')\geq\ell(\bfp\restriction Y'')=\alpha''$.
Then, from~:
\[
\alpha=\ell(\bfp)\geq\ell(\bfp\restriction X')+\ell(\bfp\restriction X'')
=\alpha'+\ell(\bfp\restriction X'')
\geq\alpha'+\alpha''=\alpha
\]
it follows that  $\alpha'+\ell(\bfp\restriction X'')=\alpha'+\alpha''$, and therefore $\ell(\bfp\restriction X'')=\alpha''$.
\end{proof}

\begin{lem}\label{lem:length-init}
Consider a WPO $\bfp=(X,\leq)$ of cardinality $\kappa$.
The length of $\bfp$ is equal to $\kappa$ \iff\
every proper initial segment of $\bfp$ has cardinality less than $\kappa$,
\iff\ for every vertex $x\in X$, the cardinality of $\{X\not\geq x\}$ is less than $\kappa$. 
\end{lem}

\begin{proof}
The second assertion is equivalent to the first one since the collection of $\{X\not\geq x\}$'s is coifinal in the collection of proper initial segments.
As for this second equivalence, recall from~\eqref{eq:induction} that~:
\[
\ell(\bfp)=\sup^+_{x\in X}\ell(\{X\not\geq x\})\geq\sup^+_{x\in X_i}|\{X_i\not\geq_i x\}|
\]
given that the length of a poset is equipotent with its vertex set.%
\footinvis{
This is an attained quantity.
}
\end{proof}

\begin{lem}\label{lem:increase-varphi}
Given a set $X$, of infinite cardinaliy $\kappa$, let us consider finitely many WQO $\bfp_1$, $\dots$, $\bfp_n$ on X of respective lengths $\alpha_1$, $\dots$, $\alpha_n$. 
If $\beta_1$, $\dots$, $\beta_n$ are ordinals of cardinality $\kappa$ such that $\alpha_1\leq \beta_1$, $\dots$, $\alpha_n\leq\beta_n$,
then there are a supserset $Y$ of $X$ and $n$ WQO  $\bfq_1$, $\dots$, $\bfq_n$ on $Y$ such that~: 
\[
\bfq_1\restriction X=\bfp_1,\dots,\bfq_n\restriction X=\bfp_n
\]
and~: 
\[
\ell(\bfq_1)=\beta_1, \dots, \ell(\bfq_n)= \beta_n.
\]
In particular~:
\[
\ell(\bfp_1\sqcap\dots\sqcap\bfp_n)\leq \ell(\bfq_1\sqcap\dots\sqcap\bfq_n).
\]

Furthermore, if the $\bfp_i$'s are well orders one can choose the $\bfq_i$'s to be well orders too. 
\end{lem}

\begin{proof}
For $i\in\{1,\dots,n\}$, let  $\leq_i$ denote the ordering of $\bfp_i$. 
First observe that it can be assumed that there is a $j$ such that $\beta_i=\alpha_i$ for each $i\not=j$.
Indeed the general case follows from the succession of $n$ applications of this particular case.
And, without loss of generality, $j$ can be assumed to be equal to $1$.
So let us assume that $\beta_i=\alpha_i$ for every $i\not=1$. 

Let $\gamma_1:=(-\alpha_1+ \beta_1)$,
so that $\beta_1=\alpha_1+\gamma_1$. 
Let $Z$ be  a set of cardinality $|\gamma_1|$ disjoint from $X$ and let $Y:= X\cup Z$.  
Given a well ordering $\preccurlyeq$ of type $\gamma_1$ on $Z$,
consider the poset $\bfq_1:=\bfp_1+(Z,\preccurlyeq)$.
Thus $\bfq_1\restriction X=\bfp_1$ and, clearly, $\ell(\bfq_1)=\ell(\bfq_1\restriction X)+\gamma_1=\alpha_1+\gamma_1=\beta_1$.
If $n=1$, then the proof is complete.
So assume that $n\geq 2$.

Consider $i\in\{2,\dots,n\}$.
Letting $\delta_i:=(-|\gamma_1|)+\alpha_i$, 
so that $|\gamma_1|+\delta_i=\alpha_i$,
consider, with Lemma~\ref{lem:increase-varphi}, 
an initial segment $X_i$ of $\bfp_i$ such that $\ell(\bfp_i\restriction X_i)=|\gamma_1|$ and $\ell(\bfp_i\restriction X\bs X_i)=\delta_i$.
In particular $\{X_i\not\geq_i x\}$ has cardinality less than $|\gamma_1|$ for each $x\in X_i$ (Lemma~\ref{lem:length-init}).
Then consider the lexicographical product poset $\bfr_i:=(\{0,1\},\leq)\multordby(\bfp_i\restriction X_i)$
on $\{0,1\}\times X_i$.
Observe that $\ell(\bfr_i)=|\gamma_1|$ because of Lemma~\ref{lem:length-init}.
Indeed 
for each $(\varepsilon,x)\in\{0,1\}\times X_i$,
$\{\{0,1\}\times X_i\not\geq_{\bfr_i}(\varepsilon,x)\}\subseteq\{0,1\}\times\{X_i\not\geq_i x\}\cup\{(0,x)\}$ has cardinality less than $|\gamma_1|$.

Now given a bijection $f:X_i\to Z$, 
let us consider a poset $\bfq_i=(Y,\leq_i')$, 
on $Y=X\dot\cup Z=X_i\dot\cup(X\bs X_i)\dot\cup Z$ such that~:
\begin{enumerate}
\item
$\bfq_i\restriction X=\bfp_i$~;
\item
$\bfq_i\restriction X_i\dot\cup Z$ be isomorphic with $\bfr_i$ through
$(0,x)\mapsto f(x)$ and $(1,x)\mapsto x$~;
\item
$X_i\dot\cup Z$ be an initial segment of $\bfq_i$. 
\footinvis{
Observe that it is not required that, for each $x\in X_i$ and $y\in X\bs X_i$~: $f(x)\leq y\Lra x\leq y$.
(and $f(x)\geq y$ \iff\ $x\geq y$, which never occurs).
}
\end{enumerate}
Such a poset can be obtained, starting from $\bfp_i$,
by substituting a two vertex linear order for each element of $X_i$. 
We claim that $\ell(\bfq_i)=\beta_i$.
Indeed~:
\[
\alpha_i=\ell(\bfp_i)
=\ell(\bfq_i\restriction X)
\leq
\ell(\bfq_i)\leq\ell(\bfq_i\restriction(X_i\dot\cup Z))+\ell(\bfq_i\restriction(X\bs\ X_i))=|\gamma_1|+\delta_i=\alpha_i=\beta_i.
\]
\end{proof}


\nouvellepage
\subsection{Conventions regarding ordinal operations}\label{section:specific-conven-ordin-oper}


For each  finitary operation $\phi$ on ordinals,
we may consider terms of the form
$\phi_\#(\cdots)$ of which the arguments are ordinals and underlined ordinals.
Such a term denotes the least ordinal strictly greater than 
the evaluation of the expression obtained by replacing $\phi_\#$ by $\phi$
and each argument by a non-greater ordinal,
by a lesser one if this argument is underlined~;
\eg~:
\begin{multline*}
\phi_\#(\alpha_1,\dots,\alpha_m,\ul{\beta_1},\dots,\ul{\beta_n}):=
\sup^+\{
\phi(\alpha_1',\dots,\alpha_m',\beta_1',\dots,\beta_n'):
\\
\alpha_1'\leq\alpha_1,\dots,\alpha_m'\leq\alpha_m,
\beta_1'<\beta_1,\dots,\beta_n'<\beta_n
\}.
\end{multline*}
Notice that if no argument is underlined and $\phi$ is $\leq$-increasing in each variable,
then $\phi_\#(\cdots)=\phi(\cdots)+1$.

Now we introduce a finitary operation $\ul\phi$ on ordinals with the same arity as $\phi$~:
\begin{align*}
\ul\phi(\alpha_1,\dots,\alpha_n)&:=\phi_\#(\ul{\alpha_1},\dots,\ul{\alpha_n})\\
&=\sup^+\{\phi(\alpha_1',\dots,\alpha_n'):\alpha_1'<\alpha_1,\dots,\alpha_n'<\alpha_n\}\\
&:=\sup\{\phi(\alpha_1',\dots,\alpha_n')+1:\alpha_1'<\alpha_1,\dots,\alpha_n'<\alpha_n\}.
\end{align*}

The following observations are easily checked~:

\begin{itemize}
\item If $\phi$ is an associative  binary operation, then $\ul\phi$ is associative.
\item If $\phi_n$ is a  $n$-ary operation obtained from an associative binary  operation  $\phi_2$, 
then $\ul{\phi_2}$ is associative and $\ul{\phi_n}$ is obtained from $\ul{\phi_2}$. 
\end{itemize}

In particular the expression $\alpha_1\uloplus\cdots\uloplus\alpha_n$ is not ambiguous~;
namely~:
\[
\alpha_1\uloplus\cdots\uloplus\alpha_n=
\sup^+\{\alpha_1'\oplus\cdots\oplus\alpha_n:
\alpha_1'\leq\alpha_1,\dots,\alpha_m'\leq\alpha_m\}.
\]

Observe that an ordinal $\alpha$ is {\em indecomposable},
\ie, it is not the sum of two lesser ordinals \iff\ $\alpha\uloplus\alpha=\alpha$.
In particular every initial ordinal is indecomposable.
We shall need the following distributivity property~:
\begin{equation}\label{eq:indec-distrib}
\alpha\uloplus\alpha=\alpha\Ra
\alpha\multordby(\beta\oplus\gamma)
=
(\alpha\multordby\beta)\oplus(\alpha\multordby\gamma).
\end{equation}


\section{Proof of Theorem~\ref{theo:main}}\label{section:proof}

Let us intoduce the following finitary operations on ordinals~:
\[
\varphi^+_P(\alpha_1,\dots,\alpha_n):=\sup^+\{\ell(\bfp_1\sqcap\dots\sqcap\bfa_n):\bfa_1,\dots,\bfa_n\text{~WPO, }\ell(\bfa_1)=\alpha_1,\dots,\ell(\bfa_n)=\alpha_n\}.
\]
Relativizing this operation to WPO's that are linear, \ie\ to WLO's, we also let~:
\[
\varphi^+_L(\alpha_1,\dots,\alpha_n):=\sup^+\{\ell(\bfa_1\sqcap\dots\sqcap\bfa_n):\bfa_1,\dots,\bfa_n\text{~WLO, }\tau(\bfa_1)=\alpha_1,\dots,\tau(\bfa_n)=\alpha_n\}.
\]
Obviously,  $\varphi^+_P(\alpha_1,\dots,\alpha_n)=\varphi^+_L(\alpha_1,\dots,\alpha_n)=0$ whenever $\alpha_1,\dots,\alpha_n$ do not have the same cardinality. 
Also, $\varphi^+_L(\alpha)=\varphi^+_P(\alpha)=\alpha+1$ for every ordinal $\alpha$, 
and $\varphi^+_L(k,\dots, k)=\varphi^+_P(k, \dots, k )=k+1$ if $k <\omega$, thus we will not need to consider the case $n=1$ nor the case of finite ordinals.  

Besides, $\varphi^+_L\leq\varphi^+_P$~; in fact, we shall see that equality holds. 
Indeed,   in terms of these two operations, Theorem~\ref{theo:main} can be rephrased as follows~:

\begin{theo}\label{lem:main}
$\varphi_{P}^+= \varphi^+_{L}$ and for every $n\geq 2$~:
\begin{equation}\label{eq:inters-main}
\varphi_{P}^+(\alpha_1,\dots,\alpha_n)=
\kappa\multordby(q_\kappa(\alpha_1)\otimes\dots\otimes q_\kappa(\alpha_n))
+|r_\kappa(\alpha_1)+\dots+r_\kappa(\alpha_n)|^+
\end{equation}  
provided that 
the arguments be equipotent,  of common infinite cardinality $\kappa$.
\end{theo}

In Section~\ref{section:reduction} below we reduce the general case to the particular case $n=2$.
The minoration of $\varphi^+_P$ is proved in Section~\ref{section:minoration} (Proposition~\ref{prop:minor-longdim2})
and the  majoration in Section~\ref{section:majoration} (Proposition~\ref{prop:major-longdim2:bis}).

 \subsection{The derivation of the general case from the case $n=2$}\label{section:reduction}

\subsubsection{Auxiliary operations}

For this reduction, we introduce two auxiliary operations $\widetilde\varphi_P$ and $\widetilde\varphi_L$. Let~:
\[
\widetilde\varphi_P(\alpha_1,\dots,\alpha_n):=\sup^+\{\ell(\bfa_1\sqcap\dots\sqcap\bfa_n):\bfa_1,\dots,\bfa_n\text{~WPO, }\ell(\bfa_1)<\alpha_1,\dots,\ell(\bfa_n)<\alpha_n\}
\]
and~:
\[
\widetilde\varphi_L(\alpha_1,\dots,\alpha_n):=\sup^+\{\ell(\bfa_1\sqcap\dots\sqcap\bfa_n):\bfa_1,\dots,\bfa_n\text{~WLO, }\tau(\bfa_1)<\alpha_1,\dots,\tau(\bfa_n)<\alpha_n\}.
\]

Obviously  $\widetilde\varphi_L(\alpha)=\widetilde\varphi_P(\alpha)=\alpha$ and in general $\widetilde\varphi_L\leq\widetilde\varphi_P$.

\subsubsection{The strategy of reduction to the case $n=2$}

The reduction will rely on the following observations. 
Letting $\widetilde\varphi_{P,n}$, $\widetilde\varphi_{L,n}$ denote the restrictions of these operations to $n$ variables~:
\begin{itemize}
\item
$\widetilde\varphi_{P,2}$ and $\widetilde\varphi_{L,2}$ are associative and $\widetilde\varphi_{P,n}$, $\widetilde\varphi_{L,n}$ are their extensions by associativity.
\item
$\widetilde\varphi_{P,n}$ and $\varphi^+_{P,n}$ are recoverable from one another
(and likewise for $\widetilde\varphi_{L,n}$ and $\varphi^+_{L,n}$).
\item
Letting $\theta^+$ denote the operation corresponding to the right-hand member of~\eqref{eq:inters-main},
consider the operation $\widetilde\theta$ defined from $\theta^+$
as $\widetilde\varphi_{P,n}$ is definable $\varphi^+_{P,n}$~;
then $\widetilde\theta_2$ is associative and $\widetilde\theta$ is its extension by associativity.
\item
The theorem in the case $n=2$ precisely says that $\varphi^+_{L,2}=\theta^+_2=\varphi^+_{P,2}$.
\end{itemize}
Note that $\widetilde\theta$ will not have to be explicited.

\subsubsection{The reduction}

So let us proceed to this reduction.
First observe that Lemma~\ref{lem:increase-varphi} yields~:

\begin{lem}\label{lem:monotony-restr} 
Given two finite tuples $(\alpha_1,\dots,\alpha_n)$, and $(\alpha_1',\dots,\alpha_n')$ of equipotent ordinals with respective cardinalities $\kappa$ and $\kappa'$ and such that $\alpha_1'\leq \alpha_1$, $\dots$, $\alpha_n'\leq\alpha_n$~: 
\begin{equation*}
\varphi_P^+(\alpha_1',\dots,\alpha_n')\leq \varphi_P^+(\alpha_1,\dots,\alpha_n)\text{~and }  \varphi_L^+(\alpha_1',\dots,\alpha_n')\leq\varphi_L^+(\alpha_1,\dots,\alpha_n).
\end{equation*}
\end{lem}

\begin{notat}
Given a finitary operation $\phi$ on ordinals, let $[\phi]^+$ and $[\phi]^\sim$ denote the following two operations~:
\[
[\phi]^+(\alpha_1,\dots,\alpha_n)=
\begin{cases}
\widetilde\phi(\alpha_1+1,\dots,\alpha_n+1)
&\text{if the $\alpha_i$'s are equipotent}
\\
0
&\text{otherwise}
\end{cases}
\]
and~:
\begin{align*}
[\phi]^\sim(\alpha_1,\dots,\alpha_n)
&=\sup^+\{\alpha:\exists \alpha_1'<\alpha_1,\dots,\alpha_n'<\alpha_n\ \alpha<\phi(\alpha_1',\dots,\alpha_n')\}
\\
&=\sup\{\phi(\alpha_1',\dots,\alpha_n'):\alpha_1'<\alpha_1,\dots,\alpha_n'<\alpha_n\}.
\end{align*}
\end{notat}

\begin{cor}\label{cor:plus-tilde}
Each operation $\varphi^+_P$ and $\widetilde\varphi_P$, resp. $\varphi^+_L$ and $\widetilde\varphi_L$,  can be recovered from the other one as follows~: 
$\varphi^+_P=[\widetilde\varphi_P]^+$, $\varphi^+_L=[\widetilde\varphi_L]^+$, 
$\widetilde\varphi_P=[\varphi^+_P]^\sim$ and $\widetilde\varphi_L=[\varphi^+_L]^\sim$.
\end{cor}



\begin{cor} 
For each $n$,   
$\varphi^+_{P,n}=\varphi^+_{L,n}$  is equivalent to
$\widetilde\varphi_{P,n}= \widetilde\varphi_{L,n}$. 
\end{cor}

Let us now come to the associtivity relations.

\begin{lem}\label{lem:associativity}
For any positive integers $m$ and $n$~:
\begin{equation}
\label{ineq:varphiplusPmn}
\widetilde\varphi_P(\alpha_1,\dots,\alpha_m,\beta_1,\dots,\beta_n)
\leq
\widetilde\varphi_P(
\widetilde\varphi_P(\alpha_1,\dots,\alpha_m),
\widetilde\varphi_P(\beta_1,\dots,\beta_n)
)
\end{equation}
and, dually~:
\begin{equation}
\label{ineq:varphiplusLmn}
\widetilde\varphi_L(\alpha_1,\dots,\alpha_m,\beta_1,\dots,\beta_n)
\geq
\widetilde\varphi_L(
\widetilde\varphi_L(\alpha_1,\dots,\alpha_m),
\widetilde\varphi_L(\beta_1,\dots,\beta_n)
).
\end{equation}
\end{lem}

\begin{proof}
For Inequality~\eqref{ineq:varphiplusPmn}, just write each $\bfp_1\sqcap\dots\sqcap\bfp_m\sqcap\bfq_1\sqcap\dots\sqcap\bfq_n$ as 
$(\bfp_1\sqcap\dots\sqcap\bfp_m)\sqcap(\bfq_1\sqcap\dots\sqcap\bfq_n)$.

Let us then handle Inequality~\eqref{ineq:varphiplusLmn}.
Consider the case of equipotent $\alpha_i$'s and $\beta_j$'s.
Observe that, according to Lemma~\ref{lem:trouverunnom} below~:
\[
\ell(\bfa\sqcap\bfb)\leq
\ell(
\underbrace{\bfa_1\sqcap\dots\sqcap\bfa_m\sqcap\bfb_1\sqcap\dots\sqcap\bfb_n}_{\sqsubseteq\bfa\sqcap\bfb}
)
<
\widetilde\varphi_L(\alpha_1,\dots,\alpha_m,\beta_1,\dots,\beta_n).
\]
Thus $\widetilde\varphi_L(\alpha_1,\dots,\alpha_m,\beta_1,\dots,\beta_n)$ is a strict upper bound of the lengths of the intersections of two well orderings of types less than 
$\widetilde\varphi_L(\alpha_1,\dots,\alpha_m)$ and 
$\widetilde\varphi_L(\beta_1,\dots,\beta_n)$ respectively, and therefore it is greater than or equal to 
$\widetilde\varphi_L(
\widetilde\varphi_L(\alpha_1,\dots,\alpha_m),
\widetilde\varphi_L(\beta_1,\dots,\beta_n)
)$, that is the least of these strict upper bounds.
%
\end{proof}

\begin{lem}\label{lem:trouverunnom}
For any two equipotent well orders $\bfa$ and $\bfb$ with the same vertex set $X$ and of respective types~:
\[
\tau(\bfa)=\alpha<\widetilde\varphi_L(\alpha_1,\dots,\alpha_m)
\text{~and }
\tau(\bfb)=\beta<\widetilde\varphi_L(\beta_1,\dots,\beta_n)
\]
there are   well orders $\bfa_i$'s and $\bfb_j$'s on $X$ such that~:
\[
\tau(\bfa_i)<\alpha_i
\text{~and }
\tau(\bfb_j)<\beta_j
\]
and 
\[
\bfa_1\sqcap\dots\sqcap\bfa_m\sqsubseteq\bfa
\text{~and }
\bfb_1\sqcap\dots\sqcap\bfb_n\sqsubseteq\bfb.
\]
\end{lem}

\begin{proof}
Since $\alpha$ is not a strict upper bound of the set of which $\widetilde\varphi_L(\alpha_1,\dots,\alpha_m)$ is defined to be the least strict upper bound, it is less than or equal to an element of that set.
So there are well orders $\bfa_i'$'s on a common vertex set $X'$ with $\tau(\bfa_i')<\alpha_i$ and
such that $\ell(\bfa_1'\sqcap\dots\sqcap\bfa_m')\geq\alpha$.
Then consider a linear extension $\bfa'$ of $\bfa_1'\sqcap\dots\sqcap\bfa_m'$ of type at least $\alpha$, and 
consider $X''\subseteq X'$ such that $\tau(\bfa'\restriction X'')=\alpha$.
Then consider the bijection $f:X''\to X$ that is an isomophism from $\bfa'\restriction X''$ to $\bfa$,
and let each $\bfa_i$ be the well order 
on $X$ onto which $f$ is an isomorphism from $\bfa_i'\restriction X''$.
Incidentally note that indeed $\bfa_1\sqcap\dots\sqcap\bfa_m=\bfa$.

Perform likewise \wrt\ $\bfb$.
\end{proof}

From Lemma~\ref{lem:associativity} a straightforward induction yields~: 

\begin{cor}  
If $\widetilde \varphi_{P,2}= \widetilde \varphi_{L,2}$  then  
$\widetilde \varphi_{P,n}=\widetilde \varphi_{L,n}$ for every integer $n\geq 2$. 
\end{cor} 

In order to complete our derivation of the general case from the case $n=2$, it remains to check~: 

\begin{lem} 
If~\eqref{eq:inters-main} in Theorem~\ref{lem:main} holds for $n=2$, then it holds for any $n$.
\end{lem}

\invis{
Cette preuve est \`a revoir.
En particulier, nous semblions, dans nos notes, implicitement, dire qu'on n'aurait pas \`a faire de v\'erification si les expressions explicites.
V\'erifier ce qu'il en est~!
Semble d\'esormais OK.
}

\begin{proof} 
Consider the finitary operation $\widetilde\theta:=[\theta^+]^\sim$ on ordinals defined from~: 
\[
\theta^+:(\alpha_1,\dots,\alpha_n)\longmapsto
\begin{cases}
\kappa\multordby(q_\kappa(\alpha_1)\otimes\dots\otimes q_\kappa(\alpha_n))+|r_\kappa(\alpha_1)+\dots+r_\kappa(\alpha_n)|^+
&\text{if $|\alpha_1|=\dots=|\alpha_n|$, let $\kappa$}
\\
0
&\text{otherwise}
\end{cases}
\]
($n\geq 2$).
As above $\theta^+=[\widetilde\theta]^+$,
because $\theta^+$ satisfies the monotonicity condition of Lemma~\ref{lem:monotony-restr}. 

To infer $\varphi^+=\theta^+$ from $\varphi_2^+=\theta_2^+$ (by $\varphi^+$ here, we mean $\varphi_R^+$ or $\varphi_L^+$),
it suffices to check that $\widetilde\theta_2$ is associative with $\widetilde\theta_n$ equal to its extension by associativity to $n$ variables.
Given that $\theta^+$ is commutative, 
it suffices to check that $[\widetilde\theta_2(\widetilde\theta_n,\id_1)]^+=\theta_{n+1}^+$,
which we do now.
Below we assume that the arguments are equipotent (written $\alpha_i\sim\beta$) of cardinal $\kappa$,
and we just write $q$ and $r$ for $q_\kappa$ and $r_\kappa$.
\begin{align*}
&[\widetilde\theta_2(\widetilde\theta_n,\id_1)]^+(\alpha_1,\dots,\alpha_n,\beta)
\\ 
&
=\widetilde\theta_2(\widetilde\theta_n(\alpha_1+1,\dots,\alpha_n+1),\beta+1)
\\
&
=\widetilde\theta_2(\theta_n^+(\alpha_1,\dots,\alpha_n),\beta+1)
\\
&
=\sup\{\theta^+_2(\alpha,\beta):\alpha<\theta_n^+(\alpha_1,\dots,\alpha_n)\}
\\
&
=\sup\{\theta^+_2(
\kappa\multordby(q(\alpha_1)\otimes\dots\otimes q(\alpha_n))+\gamma,
\kappa\multordby q(\beta)+r(\beta)
):
\gamma\sim r(\alpha_1)+\dots+r(\alpha_n)
\}
\\
&
=\sup\{
\kappa\multordby((q(\alpha_1)\otimes\dots\otimes q(\alpha_n))\otimes q(\beta))+
|\gamma+r(\beta)|^+:
\gamma\sim r(\alpha_1)+\dots+r(\alpha_n)
\}
\\
&
=\kappa\multordby(q(\alpha_1)\otimes\dots\otimes q(\alpha_n)\otimes q(\beta))+|r(\alpha_1)+\dots+r(\alpha_n)+r(\beta)|^+
\\
&
=\theta^+_{n+1}(\alpha_1,\dots,\alpha_n,\beta).
\end{align*}
\end{proof}

\subsubsection{Strategy for the case $n=2$}

We shall prove~:

\begin{prop}\label{prop:n=2}
Given two equipotent ordinals $\alpha$ and $\beta$, of cardinality $\kappa$~:
\begin{equation}
\label{ineqs:varphiplusPL2}
\varphi^+_P(\alpha,\beta)\leq\kappa\multordby(q_\kappa(\alpha)\otimes q_\kappa(\beta))+|r_\kappa(\alpha)+r_\kappa(\beta)|^+\leq\varphi^+_L(\alpha,\beta).
\end{equation}
\end{prop}

As a consequence, we shall get~: 

\begin{cor}
\[
\varphi^+_{P,2}=\varphi^+_{L,2}
\]
and
\[
\varphi_{P}^+(\alpha,\beta)=\kappa\multordby(q_\kappa(\alpha)\otimes q_\kappa(\beta))+|r_\kappa(\alpha)+r_\kappa(\beta)|^+
\] 
if the arguments are equipotent, of common cardinality $\kappa$. 
\end{cor}

The left-hand inequality in Proposition~\ref{prop:n=2} is Proposition~\ref{prop:major-longdim2:bis} and the right-hand one is 
 Proposition~\ref{prop:minor-longdim2}.


\subsection{Minoration}\label{section:minoration} 

The proof of the minoration of $\varphi^+$ is performed in two steps.
We first consider the case of two arguments that are multiples of their common cardinality,
and then we reduce the general case to this one.
 
\subsubsection{Minoration for arguments multiples of their cardinality}\label{section:minoration-partic-case} 
 
Given an infinite cardinal $\kappa$ viewed as an initial ordinal,
and two order types $\alpha$ and $\beta$ of cardinality at most $\kappa$,
let us consider three posets $\bfk$, $\bfa$ and $\bfb$ 
of respective order types $\alpha$, $\beta$ and $\kappa$,
and vertex sets $K$, $A$ and $B$.
Let $\bfp=(P,\leq)$ denote the cartesian product $(\bfk\multordby\bfa)\times(\bfk\multordby\bfb)$.
Thus $P=(K\multsetby A)\times (K\multsetby B)$.

We are interested in orders of the form $\bfp\restriction R$,
and in particular in those for wich $R\subseteq(K\multsetby A)\times (K\multsetby B)$ is the graph of a bijection between $K\multsetby A$ and $K\multsetby B$.
Indeed we shall see (Corollary~\ref{cor:minor-multiples}) that if $R$ is mixing (see below) 
and $\alpha$ and $\beta$ are types of well orders, then 
$\bfp\restriction R$ has length at least $\kappa\multordby(\alpha\otimes\beta)$ and is an intersection of two well orders of respective types $\kappa\multordby\alpha$ and  $\kappa\multordby\beta$.

Given $R\subseteq P$, let us say that $R$ is {\em functional}, resp. {\em partial functional},
if it is the graph of a function, resp. partial function of $K\multsetby A$ to $K\multsetby B$,
in other words if each vertical section $R\cap(\{(k,a)\}\times(K\multsetby B))$ has exactly one element, resp.
at most one element. 
Say that $R$ is {\em co-functional}, resp. {\em partial co-functional}
if each horizontal section $R\cap((K\multsetby A)\times\{(k,b)\})$ has exactly one element, resp.
at most one element.
Then say that $R$ is {\em bi-functional}, resp. {\em partial bi-functional}
if it is both fonctional and co-functional, resp. both partial functional and partial co-functional.
Observe that $R$ is bi-functional \iff\ this is the graph of a bijection between $K\multsetby A$ and $K\multsetby B$, and that it is partial bi-functional \iff\ this is the graph of a one-to-one partial function.

For $R\subseteq P$, and $(a,b)\in A\times B$, we let~:
\[
R_a^b:=\{(k_1,k_2):((k_1,a),(k_2,b))\in R\}
\subseteq K\times K.
\]

Note that, if $R$ is partial functional, resp. partial co-functional, 
then so is each $R_a^b$ (with the obvious extensions of the definitions).

Say that $R$ is {\em mixing} if each $R_a^b$ has cardinality $\kappa$.

\begin{lem}\label{lem:propri-mixing}
Consider $R\subseteq(K\multsetby A)\times (K\multsetby B)$.
\begin{enumerate}
\item
If $R$ is partial bi-functional then  the poset $\bfp\restriction R$
is the intersection of two orders of types at most $\kappa\multordby\alpha$ and $\kappa\multordby\beta$.
\item
If $R$ is partial functional and mixing then  the poset $\bfp\restriction R$
has an (edge)-extension of type $\kappa\multordby(\alpha\times\beta)$.
\end{enumerate}
If $R$ is partial bi-functional and mixing then  the poset $\bfp\restriction R$ 
is the intersection of two orders of types $\kappa\multordby\alpha$ and $\kappa\multordby\beta$
and has an (edge)-extension of type $\kappa\multordby(\alpha\times\beta)$.
\end{lem}

\begin{proof}\
\begin{enumerate}
\item
The ordering of $\bfp$ is the intersection of the two lexicographical products of $\bfk\multordby\bfa$ and $\bfk\multordby\bfb$ (the left one and the right one)%
\footinvis{
Pr\'eciser les trois produits en sections de pr\'eliminaires~?
}
.
If $R$ is partial bi-functional then the restrictions to $R$ of these two lexicographical products have types at most $\kappa\multordby\alpha$ and  $\kappa\multordby\beta$.
They have types exactly $\kappa\multordby\alpha$ and  $\kappa\multordby\beta$ if in addition $R$ is mixing.
\item
Consider the mapping
$((k_1,a),(k_2,b))\longmapsto(k_1,(a,b))$.
Its restriction to any partial functional $R$ is increasing from $\bfp\restriction R$ to $\bfk\multordby(\bfa\times\bfb)$.
If, in addition, $R$ is mixing, then the image of $R$ has type $\kappa\multordby(\alpha\times\beta)$. 
\end{enumerate}
\end{proof}

\begin{lem}\label{lem:exists-bi-functional-mixing}
There exists a mixing bi-functional $R\subseteq P$.
\end{lem}

\begin{proof}
Consider two partitions $(K_a:a\in A)$ and $(K^b:b\in B)$ of $K$ into classes of size $\kappa$ each.
For each $(a,b)\in A\times B$, consider the graph $G_a^b\subseteq K\times K$ of a bijection between $K^b$ and $K_a$.
Then let~:
\[
R:=\{((k_1,a),(k_2,b)):(k_1,k_2)\in G_a^b:a\in A,b\in B\}\subseteq P.
\]
Then observe that $R$ is bi-functional and mixing (with $R_a^b=G_a^b$). 
\end{proof}

In the case $\alpha$ and $\beta$ are types of well orders, these lemma yield~:

\begin{cor}\label{cor:minor-multiples}
Consider an infinite cardinal $\kappa$ and two ordinals $\alpha$ and $\beta$ of cardinality at most $\kappa$.
Then there is a WPO $\bfr$ of length at least $\kappa\multordby(\alpha \otimes \beta)$
and that is the intersection of two well orders of types $\kappa\multordby \alpha$ and $\kappa\multordby\beta$ respectively.
In particular~:
\begin{equation*}
\varphi^+_L(\kappa\multordby \alpha,\kappa\multordby\beta)>\kappa\multordby(\alpha \otimes \beta).
\end{equation*}
\end{cor}

\begin{proof}
Consider $\bfr:=\bfp\restriction R$ for $R$ bi-functional and mixing, as in Lemma~\ref{lem:propri-mixing}.
Thus,
on the one hand $\bfr$ is the intersection of two well orders of types $\kappa\multordby \alpha$ and $\kappa\multordby\beta$ respectively,
and on the other hand it has an edge-extension of isomorphy type $\kappa\multordby(\alpha\times \beta)$, so that $\ell(\bfr)\geq\ell(\kappa\multordby(\alpha\times \beta))$.
Besides~:
\[
\ell(\kappa\multordby(\alpha\times \beta))\geq\kappa\multordby\ell(\alpha\times \beta)=\kappa\multordby(\alpha\otimes\beta).
\]
The last equality is Carruth's result, \cf~\eqref{eq:dejongh}, while the inequality is easy to check~:
an extension of $\kappa\multordby(\alpha\times \beta)$
can be obtained as the ordinal product of $\kappa$ by any extension of $\alpha\times \beta$.
\end{proof}

\subsubsection{Minoration. General case}\label{section:minoration-gen-case} 

\begin{lem}\label{lem:decompinver}
Consider ordinals $\alpha_1$, $\dots$, $\alpha_n$ and $\beta_1$, $\dots$, $\beta_n$ 
such that $|\alpha_1|=|\beta_1|$, $\dots$,  $|\alpha_n|=|\beta_n|$.
Then (observe the ordering of the $\beta_i$'s)~:
\begin{equation} \label{eq:decompinver}
\varphi_L^+(\alpha_1+\dots+\alpha_n,\beta_n+\dots+\beta_1)\geq
\varphi_L^+(\alpha_1,\beta_1)\uloplus\dots\uloplus\varphi^+_L(\alpha_n,\beta_n).
\end{equation}
\end{lem}

\dessins{dim2-wo-lem-decompinver}{dim2-wo-lem-decompinver}{Proof of Lemma~\ref{lem:decompinver}.
Noter que l'indexation des $\beta_i$ ne correspond pas \`a celle de l'\'enonc\'e de cette version, mais \`a celle de la version originale
}

\begin{proof} 
For each $i\in\{1,\dots,n\}$, 
let $X_i$, 
be $n$ pairwise disjoint sets of cardinalities $\kappa_i:=|\alpha_i|=|\beta_i|$~;
let $\bfa_i$ and $\bfb_i$ be two well ordered sets of type $\alpha_i$ and $\beta_i$ with vertex set  $X_i$ and let $\bfp_i:=  \bfa_i\sqcap \bfb_i$. 
The lexicographical sums $\bfa:=\bfa_1+ \dots+ \bfa_n$ and $\bfb:= \bfb_n+ \dots + \bfb_1$ on $X:=\cup_{1\leq i\leq n}X_{i}$ have order type $\alpha:=\alpha_1+\dots+\alpha_n$ and $\beta:=\beta_n+\dots+\beta_1$ and the ordered set  $\bfp:=\bfa\sqcap \bfb $ is the direct sum $\bfp_1\uplus\dots \uplus \bfp_n$. Hence, according to de Jongh-Parikh formula, $\ell(\bfp)=\ell(\bfp_1)\oplus\dots\oplus\ell(\bfp_n)$. 

Summing up, the typical member, namely $\ell(\bfp_1)\oplus\dots\oplus\ell(\bfp_n)$, of the set of ordinals of which the right-hand side of~\eqref{eq:decompinver} is the least strict upper bound belongs to the set of ordinals $\ell(\bfp)$ of which the 
left-hand side of~\eqref{eq:decompinver} is the least strict upper bound.
Inequality~\eqref{eq:decompinver} between these least strict upper bounds follows.
\invis{See Figure~\ref{dim2-wo-lem-decompinver}.}
\end{proof}


\begin{prop}\label{prop:minor-longdim2}
Given two equipotent infinite ordinals $\alpha$ and $\beta$, of cardinality $\kappa$~:
\begin{equation}\label{eq:prop-minor-longdim2}
\varphi_L^+(\alpha,\beta)\geq\kappa\multordby(q_\kappa(\alpha)\otimes q_\kappa(\beta))+|r_\kappa(\alpha)+r_\kappa(\beta)|^+.
\end{equation}
\end{prop}

\begin{proof}
Observe that,
for any ordinals $\gamma$ equipotent to $r_\kappa(\alpha)$ and $\delta$ equipotent to $r_\kappa(\beta)$~:
\begin{itemize}
\item 
$\alpha=\delta+\kappa\multordby q_\kappa(\alpha)+r_\kappa(\alpha)$, and 
\item 
$\beta=\gamma+\kappa\multordby q_\kappa(\beta)+r_\kappa(\beta)$,
\end{itemize}
and then invoke Lemma~\ref{lem:decompinver} with these two three-block decompositions 
\invis{(see Figure~\ref{longdim2-absorption})}
to get~:
\[
\varphi_L^+(\alpha,\beta)
\geq
\varphi_L^+(\delta,r_\kappa(\beta))
\uloplus
\varphi_L^+(\kappa\multordby q_\kappa(\alpha),\kappa\multordby q_\kappa(\beta))
\uloplus
\varphi^+_L(r_\kappa(\alpha),\gamma)
\]
which, given Corollary~\ref{cor:minor-multiples} (and the definition of $\uloplus$), yields~:
\[
\varphi_L^+(\alpha,\beta)
>
\delta
\uloplus
\kappa\multordby(q_\kappa(\alpha)\otimes q_\kappa(\beta))
\uloplus
\gamma.
\]
Then observe that the the least strict upper bound of the right-hand member of this last inequality is equal to right-hand member of~\eqref{eq:prop-minor-longdim2}.
\end{proof}


\subsection{Majoration}\label{section:majoration}

For any ordinals $\alpha$ and $\beta$,  let  $\varphi(\alpha,\beta)$ denote the  supremum of the length $\ell(\bfa\sqcap\bfb)$ where  $\bfa$ and $\bfb$ are two WQO on the same set and $\ell(\bfa)=\alpha$ and $\ell(\bfb)=\beta$. We shall just write $\varphi^+(\alpha,\beta)$ for $\varphi_P^+(\alpha,\beta)$.
Note that $\varphi^+(\alpha,\beta)=\varphi(\alpha,\beta)+1$ if $\varphi(\alpha,\beta)$ is realized by some pair of WQO,
and that  $\varphi^+(\alpha,\beta)=\varphi(\alpha,\beta)$ otherwise (Item~\eqref{lem:basic-longdim2:varphivavarphiplus} of Lemma~\ref{lem:basic-longdim2} below).

\begin{lem}\label{lem:basic-longdim2}
\label{lem:isoton-equipot}
Consider two ordinals $\alpha$ and $\beta$.
\begin{enumerate}
  \item\label{lem:basic-longdim2:symetrie} 
  $\varphi^+(\alpha,\beta)=\varphi^+(\beta,\alpha)$.
  \item\label{lem:basic-longdim2:nullite} 
  $\varphi^+(\alpha,\beta)=0$ \iff\ $\alpha$ and $\beta$ fail to be equipotent.
  \item\label{lem:basic-longdim2:entiers} 
  For every integer $n$, $\varphi(n,n)=n$ and $\varphi^+(n,n)=n+1$.
  \item\label{lem:basic-longdim2:obviousminor} 
  $|\alpha|=|\beta|\Ra\varphi^+(\alpha,\beta)>\max(\alpha,\beta)$.
  \item\label{lem:basic-longdim2:varphivavarphiplus}  
  $\varphi(\alpha,\beta)\leq\varphi^+(\alpha,\beta)\leq\varphi(\alpha,\beta)+1$, and $\varphi^+(\alpha,\beta)=\varphi(\alpha,\beta)+1$
  \iff\ the corresponding supremum is attained,
  \iff\ $\varphi^+(\alpha,\beta)$ is a successor ordinal.
  \item\label{lem:basic-longdim2:major} 
  $\varphi(\alpha,\beta)\leq\alpha\otimes\beta$.
%
  \item\label{lem:basic-longdim2:isoton-equipot} 
Given a third ordinal $\alpha'$~:
$
|\alpha|\leq\alpha'\leq\alpha\Ra\varphi(\alpha',\beta)\leq\varphi(\alpha,\beta)
$.
\item\label{lem:basic-longdim2:sharpplus}
$|\alpha|=|\beta|\Lgrh\varphi^+(\alpha,\beta)=\varphi_\#(\alpha,\beta)$.
\end{enumerate}
\end{lem}

\begin{proof}
Items~\eqref{lem:basic-longdim2:symetrie}, \eqref{lem:basic-longdim2:nullite}, \eqref{lem:basic-longdim2:entiers} and~\eqref{lem:basic-longdim2:varphivavarphiplus} are obvious. 
Below $\bfa$ and $\bfb$ denote well orders of type $\alpha$ and $\beta$ with the same vertex set.
\begin{enumerate}
  \item[\eqref{lem:basic-longdim2:obviousminor}] 
Each of $\bfa$ and $\bfb$ is an extension of $\bfa\sqcap\bfb$.
  \item[\eqref{lem:basic-longdim2:major}]  
Note that, as mentionned in the introduction, the poset $\bfa\sqcap\bfb$ is embeddable into the direct product $\bfa\otimes\bfb$. 
Thus $\ell(\bfa\sqcap\bfb) \leq \ell(\bfa\otimes\bfb)=\ell(\bfa)\otimes\ell(\bfb)=\alpha\otimes \beta$  by~\eqref{eq:restriction} and de Jongh-Parikh formula~\eqref{eq:dejongh}.
\item[\eqref{lem:basic-longdim2:isoton-equipot}]  
\invis{
(See Figure~\ref{isoton-equipot}.)
Let us assume that $\alpha$ is infinite and equipotent with $\beta$ (otherwise the conclusion is obvious).
We vertex-extend each intersection $\bfa'\sqcap\bfb$ of two well-orders of types $\alpha'$ and $\beta$ to 
an order $\bfa\sqcap\bfb$ intersection of two well-orders of types $\alpha$ and $\beta$.

To this end, consider the ordinal $\gamma:=(-\alpha')+\alpha$ such that $\alpha=\beta+\gamma$.
We plug a chain of type $\gamma$ above $\bfa'$, and insert these new vertices in the initial segment of $\bfa$ of type $|\gamma|$.

Poursuivre~\dots
}
Invoke Lemma~\ref{lem:increase-varphi}.  
\item[\eqref{lem:basic-longdim2:sharpplus}] 
Consider Item~\eqref{lem:basic-longdim2:isoton-equipot}.
\end{enumerate}
\end{proof}

\dessins{isoton-equipot}{isoton-equipot}{%
Proof of Lemma~\ref{lem:basic-longdim2}-\ref{lem:basic-longdim2:isoton-equipot}~:
$\bfa'\sqcap\bfb\hkra\bfa\sqcap\bfb$ donc (\eg\ par $OEP$) $\ell(\bfa'\sqcap\bfb)\leq\ell(\bfa\sqcap\bfb)$.
}

\begin{lem}\label{lem:indmaj}
For any ordinals $\alpha$ and $\beta$~:
\[
\varphi(\alpha,\beta)\leq\varphi_\#(\ul\alpha,\beta)\uloplus\varphi_\#(\alpha,\ul\beta).
\]
\end{lem}

\footinvis{
Noter qu'on ne saurait remplacer $\varphi$ par $\varphi^+$~: pour tout cardinal $\kappa$,
$\varphi^+(\kappa,\kappa)=\kappa+1$ tandis que $\varphi_\#^+(\ul \kappa, \kappa)\uloplus\varphi_\#^+(\kappa,\ul \kappa)=\kappa\uloplus \kappa$,
qui est \'egal \`a $\kappa$ lui-m\^eme lorsque le cardinal $\kappa$ est infini.
}

\invis{
La version originiale de ce lemme \'etait \'enonc\'ee dans le cas particulier de l'interseciton de deux bon ordres.
On a \'etendu (son acception) aux intersections de WPO conform\'ement \`a l'(ancienne) observation suivante.
On a donc substit\'e, dans la preuve, WPO of length $\alpha$ \`a bon ordre de type $\alpha$.

Ce lemme, ainsi que sa preuve (il convient juste d'ajuster suivant la Footnote~\ref{ftnte:chain-wpo:lem-indmaj})
demeure valide dans le cas o\`u la fonction $\varphi$ est relative \`a la longueur de l'intersection de deux WPO.
V\'erifier.
}

\footinvis{
On n'aura, pour le moment, pas eu \`a se soucier de l'\'eventuelle croissance de $\varphi^+$~!!!!!!!!!!!!!!!!!!!!!!!!!!!!!!!!!!!!!!!!!!!!!!!
}

\begin{proof}[Proof of Lemma~\ref{lem:indmaj}]
If $|\alpha|\neq|\beta|$ then the left-hand member of the inequality is $0$.
So assume that $|\alpha|=|\beta|$
and consider two WPO $\bfa$ and $\bfb$ on a set $X$ of respective lenghts $\alpha$ and $\beta$, 
and let $\bfp$ denote the poset intersection $\bfa\sqcap\bfb$.
Recall~\eqref{eq:induction}~:
\[
\ell\bfp=\sup^+_{x\in X}\ell\{\bfp\not\geq_\bfp x\}
\] 
while~:
\footinvis{
\`A d\'eplacer~:
Here, since $\bfa$ and $\bfb$ are total, 
\[
\{X\not\geq_\bfp x\}=\{X<_\bfa\vee<_\bfb x\}=\{X<_\bfa x\}\cup\{X<_\bfb x\}=\{X<_\bfa\wedge<_\bfb x\}\dot\cup\{X<_\bfa\wedge\geq_\bfb x\}\dot\cup\{X<_\bfb\wedge\geq_\bfa x\}
\]
\dessin{tricot}

Noter que $\bfp$ admet une cha\^\i ne cofinale \ssi\ chaque segment final non vide d'un des facteurs est cofinal dans l'autre. 
}
\[
\{X\not\geq_\bfp x\}=\{X\not\geq_\bfa x\}\cup\{X\not\geq_\bfb x\}
\]
so that~:
\[
\ell\{\bfp\not\geq_\bfp x\}\leq\ell\{\bfp\not\geq_\bfa x\}\oplus\ell\{\bfp\not\geq_\bfb x\}
\]
and then~:
\[
\ell\bfp\leq\sup^+_{x\in X}\ell\{\bfp\not\geq_\bfa x\}\oplus\ell\{\bfp\not\geq_\bfb x\}.
\] 
Now observe that the poset $\{\bfp\not\geq_\bfa x\}:=\bfp\restriction\{X\not\geq_\bfa x\}$ is the intersection of the posets%
\footinvis{
\label{ftnte:chain-wpo:lem-indmaj}
Le terme "chain" peut \^etre remplac\'e par "WPO" dans le cas o\`u on consid\`ere des intersections de WPO,
et juste en dessous, "type" peut alors \^etre remplac\'e par "length". 
} 
$\{\bfa\not\geq_\bfa x\}:=\bfa\restriction\{X\not\geq_\bfa x\}$ and $\{\bfb\not\geq_\bfa x\}:=\bfb\restriction\{X\not\geq_\bfa x\}$,
the first of which has length less than $\alpha$ and the second of which has length at most $\beta$. 
Therefore~:
\[
\ell\{\bfp\not\geq_\bfa x\}<\varphi_\#(\ul\alpha,\beta)
\]
and likewise~:
\[
\ell\{\bfp\not\geq_\bfb x\}<\varphi_\#(\alpha,\ul\beta)
\]
so that~:
\[
\ell\{\bfp\not\geq_\bfa x\}\oplus\ell\{\bfp\not\geq_\bfb x\}<\varphi_\#(\ul\alpha,\beta)\uloplus\varphi_\#(\alpha,\ul\beta).
\]
Henceforth~:
\[
\ell\bfp\leq\varphi_\#(\ul\alpha,\beta)\uloplus\varphi_\#(\alpha,\ul\beta)
\]
and finally~:
\[
\varphi(\alpha,\beta)\leq\varphi_\#(\ul\alpha,\beta)\uloplus\varphi_\#(\alpha,\ul\beta).
\]
\end{proof}

\dessins{dim2-wo-lem-indmaj}{dim2-wo-lem-indmaj}{Proof of Lemma~\ref{lem:indmaj}~:
$\{\bfp\not\geq_\bfa x\}=\{\bfa\not\geq_\bfa x\}\sqcap\{\bfb\not\geq_\bfa x\}=\{\bfa<_\bfa x\}\sqcap\{\bfb<_\bfa x\}$ and $\{X\not\geq_\bfp x\}$.%
\invis{
\
This picture show the situation in the case of the intersection of two linear orderings.
}
}

\begin{cor}\label{cor:kappakappa}
For every cardinal $\kappa$ (\ie\ initial ordinal)~:
\[
\varphi^+(\kappa,\kappa)=\kappa+1.
\]
\end{cor}

\begin{proof}
According to Lemma~\ref{lem:indmaj}~:
\[
\varphi(\kappa,\kappa)\leq\varphi_\#(\ul\kappa,\kappa)\uloplus\varphi_\#(\kappa,\ul \kappa)=\kappa\uloplus\kappa=\kappa
\]
so that $\varphi^+(\kappa, \kappa)\leq\varphi(\kappa, \kappa)+1\leq\kappa+1$.
The reverse inequality is straightforward.
\end{proof}

\begin{lem}\label{lem:major-sum:bis}
Given any ordinals $\alpha'$, $\alpha''$ and $\beta$~:
\begin{equation}\label{eq:major-sum:bis}
\varphi^+(\alpha'+\alpha'',\beta)\leq\varphi_\#(\alpha',\beta)\uloplus\varphi_\#(\alpha'',\beta)\leq\varphi_\#(\alpha',\beta)\uloplus|\alpha''|^+.
\end{equation}
\invis{
Il ne semble pas qu'on utilise cette seconde \'egalit\'e.
V\'erifier si, en la contournant, on n'a pas finalement compliqu\'e les choses.

In particular, when these three ordinals are equipotent~:
\begin{equation}\label{eq:major-sum-part:bis}
\varphi(\alpha'+\alpha'',\beta)\leq\varphi(\alpha',\beta)\oplus\varphi(\alpha'',\beta).
\end{equation}
}
\end{lem}

\begin{proof}
Observe that if $\alpha:=\alpha'+\alpha''$ fails to be equipotent with $\beta$,
then the left-hand member of~\eqref{eq:major-sum:bis} is $0$, in which case this inequality is trivially satisfied.
So assume that $\alpha$ and $\beta$ are equipotent.

\invis{
(See Figure~\ref{dim2-lem-major-sum}.)
}
Given two posets $\bfa$ and $\bfb$ with the same vertex set $X$ and respective lengths $\alpha$ and $\beta$,
let $\bfp:=\bfa\sqcap\bfb$.
Consider a partition of $X$ into subsets $X'$ and $X''$ of respective lengths $\alpha'$ and $\alpha''$ \wrt\ $\bfa$
(\cf~Lemma~\ref{lem:longcut}). 
Letting $\beta'$ and $\beta''$ denote the respective lengths of $X'$ and $X''$ \wrt\ $\bfb$~: 
\[
\ell(\bfp)\leq\ell(\bfp\restriction X')\oplus\ell(\bfp\restriction X'')\leq\varphi(\alpha',\beta')\oplus\varphi(\alpha'',\beta'')
<\varphi_\#(\alpha',\beta)
\uloplus
\varphi_\#(\alpha'',\beta)
\leq\varphi_\#(\alpha',\beta)\uloplus|\alpha''|^+.
\]
For the left-hand inequality, recall~\eqref{eq:restriction}~;
the second one holds by definition of $\varphi$~;
the penultimate one holds by definition of $\varphi_\#$, given that, obviously, $\beta'\leq\beta$ and $\beta''\leq\beta$~;
the right-hand one follows from $\varphi_\#(\alpha'',\beta)\leq|\alpha''|^+$.
Then~\eqref{eq:major-sum:bis} holds 
by definition of $\varphi^+$.
\footinvis{
Besides, since $\bfp$ splits into the union of the intersection of WPO of length $\alpha'$ and of a WPO of length at most $\beta$ on the one hand and of a poset of cardinal $|\alpha''|$ on the other hand~:
\[
\ell(\bfp)\leq\ell(\bfp\restriction X')\oplus\ell(\bfp\restriction X'')<\varphi_\#(\alpha',\beta)\uloplus|\alpha''|^+
\] 
so that $\varphi^+(\alpha,\beta)\leq\varphi_\#(\alpha',\beta)\uloplus|\alpha''|^+$,
which can also be inferred from $\varphi^+(\alpha,\beta)<\varphi_\#(\alpha',\beta)\uloplus\varphi_\#(\alpha'',\beta)$,
given that $\varphi_\#(\alpha'',\beta)\leq|\alpha''|^+$.
}
\footinvis{
Besides, $\ell(\bfp)\leq\varphi(\alpha',\beta')\oplus\varphi(\alpha'',\beta'')$
implies $\varphi(\alpha,\beta)\leq\varphi(\alpha',\beta')\oplus\varphi(\alpha'',\beta'')$,
which implies $\varphi(\alpha,\beta)\leq\varphi(\alpha',\beta)\oplus\varphi(\alpha'',\beta)$
when $\alpha'$ and $\alpha''$ are both equipotent with $\beta$, 
given that $\beta'\leq\beta$ and $\beta''\leq\beta$.
}
\end{proof}

\dessins{dim2-lem-major-sum}{dim2-lem-major-sum}{Proof of Lemma~\ref{lem:major-sum}.
Figure correspondant \`a une variante (plus faible) dont nous nous contentons.}

\begin{prop}\label{prop:major-longdim2:bis}
Given two equipotent infinite ordinals $\alpha$ and $\beta$, of cardinality $\kappa$~:
\begin{equation}\label{eq:major-longdim2}
\varphi_R^+(\alpha,\beta)\leq\kappa\multordby(q_\kappa(\alpha)\otimes q_\kappa(\beta))+|r_\kappa(\alpha)+r_\kappa(\beta)|^+.
\end{equation}
\end{prop}

\invis{
Recall that we drop the subscript $R$ from $\varphi_R$.
}

\begin{proof}[Proof of Proposition~\ref{prop:major-longdim2:bis}]
Given $\kappa$, we prove~\eqref{eq:major-longdim2} by induction on $(\alpha,\beta)$.
First note that,
according to Corollary~\ref{cor:kappakappa},
\eqref{eq:major-longdim2} holds for $(\alpha,\beta)=(\kappa,\kappa)$.
So let us assume that $(\alpha,\beta)>(\kappa,\kappa)$, \ie\ that 
at least one of the inequalities $\alpha\geq\kappa$ and $\beta\geq\kappa$ is strict~;
let us assume that~\eqref{eq:major-longdim2} holds
for every pair $(\alpha_1,\beta_1)$ such that $(\kappa,\kappa)\leq(\alpha_1,\beta_1)<(\alpha,\beta)$~;
and let us prove that it holds also for $(\alpha,\beta)$.
We are led to distinguish three cases.
We shall first assume that $\alpha$ or $\beta$ is not a multiple of $\kappa$~; 
then we shall assume that 
they are both multiples of $\kappa$ but that at least one of them fails to be indecomposable~;
and it will remain to consider the case of both being 
indecomposable (and in particular multiples of $\kappa$).
Let 
$\alpha=\kappa\cdot \alpha'+\gamma$ and $\beta=\kappa\cdot \beta'+\delta$,
with $\gamma<\kappa$ and $\delta <\kappa$, denote the euclidian divisions of $\alpha$ and $\beta$ by $\kappa$.
\begin{enumerate}
\item
Assume that $\alpha$ or $\beta$ is not a multiple of $\kappa$, \eg\ that  $\gamma>0$. 
In this case~:
\begin{align*}
\varphi^+(\alpha,\beta)
&
\leq
\varphi_\#(\kappa\multordby\alpha',\kappa\multordby\beta'+\delta)\uloplus|\gamma|^+
&
\text{by~\eqref{eq:major-sum:bis}}
\\
&
=
\varphi^+(\kappa\multordby\alpha',\kappa\multordby\beta'+\delta)\uloplus|\gamma|^+
&
\text{Lemma~\ref{lem:basic-longdim2}\eqref{lem:basic-longdim2:sharpplus}}
\\
&
\leq
(\kappa\multordby(\alpha'\otimes\beta')+|\delta|^+)
\uloplus|\gamma|^+
&
\text{induction assumption}
\\
&
\leq
\kappa\multordby(\alpha'\otimes\beta')+|\delta+\gamma|^+
&
\text{easily checked.}
\end{align*}
\item
Let us assume now that $\alpha$ and $\beta$ are both multiples of $\kappa$ 
and that at least one fails to be indecomposable, \eg\ $\alpha$.
Observe that $\alpha$ being a multiple of $|\alpha|$, 
its being decomposable is equivalent to $\alpha':=q_{|\alpha|}(\alpha)$ being decomposable.
Thus we assume that $\alpha'=\alpha_1'+\alpha_2'$ with $\alpha_1'$ and $\alpha_2'$ both less than $\alpha'$,
and we can even assume that, indeed, $\alpha'=\alpha_1'+\alpha_2'=\alpha_1'\oplus\alpha_2'$.
Then~:
\begin{align*}
\varphi^+(\alpha,\beta)
&
=
\varphi^+(\kappa\multordby\alpha_1'+\kappa\multordby\alpha_2',\kappa\multordby\beta')
\\
&
\leq
\varphi_\#(\kappa\multordby\alpha_1',\kappa\multordby\beta')
\uloplus
\varphi_\#(\kappa\multordby\alpha_2',\kappa\multordby\beta')
&
\text{by~\eqref{eq:major-sum:bis}}
\\
&
=
\varphi^+(\kappa\multordby\alpha_1',\kappa\multordby\beta')
\uloplus
\varphi^+(\kappa\multordby\alpha_2',\kappa\multordby\beta')
&
\text{Lemma~\ref{lem:basic-longdim2}\eqref{lem:basic-longdim2:sharpplus}}
\\
&
\leq
(\kappa\multordby(\alpha_1'\otimes\beta')+1)
\uloplus
(\kappa\multordby(\alpha_2'\otimes\beta')+1)
&
\text{induction assumption}
\\
&
=
\kappa\multordby(\alpha_1'\otimes\beta')
\oplus
\kappa\multordby(\alpha_2'\otimes\beta')+1
\\
&
=
\kappa\multordby((\alpha_1'\oplus\alpha_2')\otimes\beta')+1
&
\text{by~\eqref{eq:indec-distrib}}
\end{align*}
which is indeed~\eqref{eq:major-longdim2},
given that the two remainders are $0$.
\item
Let us now assume that
$\alpha$ and $\beta$ are both indecomposable~;
in particular $\gamma=0$ and $\delta=0$.
Recall from Lemma~\ref{lem:indmaj} that~:
\[
\varphi(\alpha,\beta)\leq\varphi_\#(\ul\alpha,\beta)
\uloplus
\varphi_\#(\alpha,\ul\beta).
\]
We claim that $\lambda:=\kappa\multordby(\alpha'\otimes\beta')$ is a strict upper bound of $\{\varphi(\alpha,\beta):\alpha_1<\alpha,\beta_1\leq\beta\}$.
With this claim $\varphi_\#(\ul\alpha,\beta)\leq\lambda$, and likewise $\varphi_\#(\alpha,\ul\beta)\leq\lambda$.
Hence, given that $\lambda$ is indecomposable~:
\[
\varphi(\alpha,\beta)\leq\varphi_\#(\ul\alpha,\beta)
\uloplus
\varphi_\#(\alpha,\ul\beta)
\leq
\lambda\uloplus\lambda=\lambda=
\kappa\multordby(\alpha'\otimes\beta')
\]
which implies~\eqref{eq:major-longdim2}.

Thus it remains to check the claim.
To this end, consider $\alpha_1<\alpha=\kappa\multordby\alpha'$ and $\beta_1<\beta=\kappa\multordby\beta'$.
\begin{itemize}
\item 
If $\alpha'=1$, then $\alpha_1<\kappa$, so that~:
\[
\varphi(\alpha_1,\beta_1)<|\alpha_1|^+\leq\kappa\leq\kappa\multordby\beta'=\kappa\multordby(\alpha'\otimes\beta').
\] 
\item 
If $\alpha'>1$, then this ordinal, which is assumed to be indecomposable, is an infinite limit ordinal.
In this case $\alpha_1\leq\kappa\multordby\alpha''<\alpha$ for some $\alpha''<\alpha'$.
Then, invoking the induction assumption for the middle inequality~:
\[
\varphi(\alpha_1,\beta_1)\leq
\varphi(\kappa\multordby\alpha'',\kappa\multordby\beta')\leq
\kappa\multordby(\alpha''\otimes\beta')+1<
\kappa\multordby(\alpha'\otimes\beta').
\]
\end{itemize}
\end{enumerate}
\end{proof}


\end{document}